\documentclass[runningheads]{llncs}  
\usepackage[T1]{fontenc}
\usepackage{graphicx}
\usepackage{booktabs}
\usepackage{amsmath,amsfonts,graphicx}
\newtheorem{algorithm}[theorem]{Algorithm}

\begin{document}
\title{New Progress in Classic Area: Polynomial Root-squaring
and Root-finding}
\author{
Victor Y. Pan\inst{} 
}
\authorrunning{
}
\institute{Department of Computer Science \\
Lehman College of the City University of New York \\
Bronx, NY 10468 USA and \\ 
 Ph.D. Programs in Mathematics  and Computer Science \\
The Graduate Center of the City University of New York \\
New York, NY 10036 USA \\ 
 victor.pan@lehman.cuny.edu \\
http://comet.lehman.cuny.edu/vpan/ \\  
}  
\maketitle

\begin{abstract}
 The {\em DLG root-squaring iterations}, due to Dandelin 1826 and rediscovered by Loba\-chevsky 1834 and Gr{\"a}ffe 1837,
 have been the main 
approach to {\em root-finding for a univariate polynomial} $p(x)$ in the 19th century and beyond, but  not so nowadays because these iterations are prone to severe numerical stability problems. Trying to avoid these problems we have found simple but novel reduction of the 
  iterations applied for Newton's inverse ratio $-p'(x)/p(x)$  to approximation of the power sums of  the zeros of 
$p(x)$ and its reverse polynomial.
The resulting polynomial root-finders  can be devised and performed independently of DLG  iterations -- based on Newton's identities or  Cauchy integrals.  In the former case the computation involve a set of leading or tailing coefficients of an input polynomial.
In the latter case we must scale the variable and increase the arithmetic  computational cost to ensure numerical stability. Nevertheless the cost is still manageable, at least for fast root-refinement, and the algorithms can be applied to a {\em black box polynomial} $p(x)$ -- given by a black box for the evaluation of the ratio $\frac{p'(x)}{p(x)}$ rather than by its coefficients. This enables important computational benefits,
 including (i) efficient recursive as well as concurrent approximation of a set of zeros of $p(x)$  or even all of its zeros, (ii) acceleration where an input polynomial can be evaluated fast, and (iii) extension to {\em approximation of the eigenvalues of a matrix  or a polynomial matrix}, being  efficient if the matrix
 can be inverted fast, e.g., is data sparse. We also recall our recent fast algorithms  for approximation of the root radii, that is, the distances to the roots from the origin or any complex value to the zeros of $p(x)$,
 and propose to apply it for fast black box initialization of polynomial root-finding by means of functional iterations such as Newton's, Ehrlich's, and Weierstrass's.
\end{abstract}

\paragraph{Keywords:} symbolic-numeric computing, computer algebra, polynomial computations, root-finding, matrix eigenvalues

{\bf 2020 Math. Subject Classification:} 65H05, 26C10, 30C15, 65H17

\section{Introduction: The State of the Art, Related Work, and Our Progress}

{\bf 1.1. Polynomial root-finding.} The venerated problem  of univariate polynomial root-finding, that is, approximation of  
 the roots of the equation $p(x)=0$ or just  {\em roots}, aka {\em zeros} of $p(x)$, for a $d$th degree  
 polynomial\footnote{To simplify subsequent notation we let $p(x)$ be monic.}
\begin{equation}\label{eqpoly}
p(x):=\sum_{i=0}^dp_ix^i=
p_d\prod_{j=1}^d(x-x_j),~p_d\neq 0,
\end{equation} 
has been central for Mathematics and Computational Mathematics for four millennia, from Sumerian times to about 1850  \cite{MB11}.   Nowadays it is  still highly important in various areas of symbolic and numerical  computing \cite[Introduction]{MP13},\cite{P22}.

{\bf 1.2. Classic root-squaring   and applications to root-finding.}
The {\em DLG root-squaring iterations} by Dandelin 1826, Lobachevsky 1834,  Gr{\"a}ffe 1837 \cite{H59},
\begin{equation}\label{eqdnd} 
 p_0(x):=p(x),~p_{h+1}(x):=(-1)^ dp_h(\sqrt x~)
p_h(\sqrt{-x}~),~h=0,1,\dots,
\end{equation} 
\begin{equation}\label{eqdnd2} 
p_h(x):=\sum_{i=0}^dp_i^{(h)}x^i=\prod_{j=1}^d(x-x_j^{(h)}),~x_j^{(h)}=x_j^{2^h},~j=1,\dots,d;~h=0,1,\dots,
\end{equation}  
have dominated polynomial root-finding in the 19-th   century and beyond.
The iterations were performed  by hand,
and people,  
called ``calculateurs''  or ``computers'', were paid for this work  \cite{MZ01}.

  Vieta's formulas  express
the ratios $(-1)^j\frac{p_{d-j}^{(h)}}{p_{d}^{(h)}}$ as the $j$th  elementary symmetric polynomials in the roots, and hence
\begin{equation}\label{eqrtrt}
-\frac{p_{d-j+1}^{(h)}}{p_{d-j}^{(h)}}\approx 
x^{(h)}_{j}=x_{j}^{2^h}
\end{equation}
if $2^h$  is large enough and if 
\begin{equation}\label{eqnndcrs} 
\max_{i>j}|x_i|< |x_{j}| <\min_{i<j}|x_i|~{\rm where}~|x_0|=+\infty|~{\rm and}~|x_{d+1}|=0.
 \end{equation}
 
 Thus the ratios $-\frac{p_{d-j+1}^{(h)}}{p_{d-j}^{(h)}}$ approximate the zeros  $x_{j}^{(h)}=x_j^{2^h}$ of the polynomial $p_h(x)$, and we can readily recover  the root radii $|x_j|$, and then the roots $x_j$ if they are real.  Various complex polynomial root-finders
extensively involve root radii (see \cite{P97},  \cite{P00}, \cite{P22}) but also apply DLG root-squaring to strengthening the isolation
  of the unit circle $\{x:~|x|=1\}$ from the zeros of $p(x)$ (see, e.g., \cite{P00}, \cite{BSSY18}, and \cite[Secs.
5.1.4, 5.3, 5.6, 5.7, 6.3, 6.6,
10.5]{P22}) and to root-finding over finite fields \cite{GHL15}.
  
We focus on approximation  (\ref{eqrtrt}) for $j=d$ and $j=1$:
\begin{equation}\label{eqxdrt}
x_d^{(h)}\approx -\frac{p_{1}^{(h)}}{p_{0}^{(h)}}
~{\rm and}~x_1^{(h)}\approx -\frac{p_{d-1}^{(h)}}{p_{d}^{(h)}}.
\end{equation}
 
 {\bf 1.3. Recovery  of complex zeros} $x_j$ of   $p(x)$ 
from the zero $x^{(h)}_
j$ of $p_{h}(x)$.  We
 can recover $x_j$ from $y_j:=x_j^{(h)}=x_j^{2^{h}}$ as follows:  check whether $p(y_j^{1/2^{h}})=0$ for each of $2^{h}$ candidates $y_j^{1/2^{h}}$ and  
weed out up to  $2^{h}-1$ ``false'' candidates. $2^h$ tests are
 expensive  if $2^{h}$ is large, but by extending the {\em descending process} of \cite{P96} we can manage with only $2h$ tests if we recursively apply the following recipe for 
$i=h-1,\dots,1,0$: given a zero 
$x^{(i+1)}_j$ of $p_{i+1}(x)$, 
select one of the two candidates 
$\pm(x^{(i+1)}_j)^{1/2}$ for being a zero of 
 $p_{i}(x)$; then decrease 
$i$ by 1 and repeat until $i$ vanishes.

{\bf 1.4. FG iterations.} By following Fiedler    \cite{F55},  Gemingani  extends DLG iterations 
in  \cite{G01} to approximate the zeros $x_j$ themselves rather than their high powers\footnote{See  other extensions of DLG iterations in \cite{BP96} and \cite{MP00}.} 
and thus to get rid of the descending process for  non-real zeros $x_j$.
 He  substitutes $2x$ for 2 on the left sides of  expressions of \cite{F55}
 for  a fixed polynomial $q(x)$, arrives at {\em Fiedler/Gemingani's (FG) iterations}: 
\begin{equation}\label{eqfdlr}
q_0(x):=q(x),~2\sqrt x q_{h+1}(x):=
q_h(\sqrt x~)p_h(\sqrt {-x}~)-q_h(\sqrt {-x}~)p_h(\sqrt x~),  
\end{equation}  
for $h=0,1,\dots$, writes  
\begin{equation}\label{eqqp'}
q(x):=xp'(x)~{\rm and}~q_h(x):=\sum_{i=0}^dq_i^{(h)}x^i,~h=0,1,\dots~{\rm (cf.~ (\ref{eqdnd2}))},
\end{equation}
and proves the following result
(see   \cite[Thm. 2]{G01}).

\begin{theorem}\label{thgmn} 
Let
\begin{equation}\label{eqseprd} 
|x_d|<\min_{j<d}|x_j|,
 \end{equation} 
  Then			
 $$\frac{q_0^{(h)}}{p_1^{(h)}}=\frac{q_h(0)}{p_h'(0)}=x_d\Big(1+O\Big(\Big|\frac{x_d}{x_{d-1}}\Big|^{2^h}\Big)\Big)~{\rm for}~h=0,1,\dots $$
\end{theorem}
 
 {\bf 1.5. Complexity and numerical   stability.}  One can reduce every DLG  or FG iteration to one or two polynomial multiplications, respectively, and then perform it  by using $O(d \log(d))$ arithmetic operations   (cf., e.g., \cite[Sec. 2.4]{P01}),   
 but the coefficients of the polynomials $p_h(x)$  change their absolute values  dramatically and unevenly as $h$ grows, implying  severe  numerical instability already for rather small integers  $h$. Realistically, one can safely apply
at most order of $\log(\log(d))$
DLG iterations or resort to
randomized {\em renormalization} of \cite{MZ01},
 the only known stabilization recipe, based on using polar 
 representation of complex numbers 
$x=|x|\exp(2\pi\sqrt{-1})$; 
back and forth transition to computations,  making them not competitive. 

{\bf 1.6. Our progress (outline).}     
  Given a function $f(x)$ and its derivative define {\em Newton's inverse ratio} 
 \begin{equation}\label{eqnwtrt}
\frac{1}{N(f(x))}:=-\frac{f'(x)}{f(x)}
 \end{equation}
  and notice that
 the two ratios in  (\ref{eqxdrt})
 are equal to the values at 0 of Newton's inverse ratios
 for $f(x)$ denoting 
  $p_k(x)$ and its 
 {\em  reverse polynomial}
 \begin{equation}\label{eqpolyrev}
p_{\rm rev}(x):=x^dp\Big(\frac{1}{x}\Big)=\sum_{i=0}^dp_ix^{d-i},~
p_{\rm rev}(x)=p_0\prod_{j=1}^d\Big(x-\frac{1}{x_j}\Big)~{\rm if}~p_0\neq 0,
\end{equation}  
obtained from  $p(x)$ by reversing the order of its coefficients.
  We combine Vieta's formulae with a well-known expression  for Newton's inverse ratio through the zeros of $p(x)$
(cf.  (\ref{eqratio})) and immediately deduce that the
 two ratios in  (\ref{eqxdrt}) are the $2^h$th power sums  
  of the zeros of $p(x)$ and of their reciprocals (see (\ref{eqdlgspwrsms})).
 These power sums  are close to the $2^h$th powers of the extremal zeros of $p(x)$
 if   assumptions (\ref{eqnndcrs}) hold for $j=1$ and $j=d$ and  if the integer $2^h$ is sufficiently large.
 
Therefore,  we can implicitly perform  $h$ DLG iterations by computing these power sums 
with all cited benefits and applications  and with the hope to avoid problems
of numerical stability, at least    for larger $h$ than in the DLG case.

Furthermore, we need neither FG iterations nor descending process to approximate non-real
 zeros of $p(x)$. We just 
 compute the ratio of the $h$th and $(h+1)$st power sums for a sufficiently large integer $h$ because these ratios converge fast  to the two extremal (that is, absolutely largest and absolutely smallest) zeros of $p(x)$.

For power sum computation, we can apply the  known algorithms  based on Newton's identities or  Cauchy integrals. In the latter case 
we scale the variable to ensure numerical stability 
at the price of some increase of arithmetic cost, but  the algorithms can be applied to a {\em black box polynomial} $p(x)$ -- given by a black box for the evaluation of the ratio $\frac{p'(x)}{p(x)}$ rather than by its coefficients, and this enables a some important computational benefits: 
 
(i) We can recursively extend this  approximation of a single zero  of $p(x)$ to  approximation of many or all its zeros.\footnote{Root-finding  for a black box polynomial with the goal of minimizing the number of evaluations of 
  the ratio $\frac{p'(x)}{p(x)}$ or  $\frac{p(x)}{p'(x)}$  has been  a  well-known research subject for years  but remained in stalemate since the paper \cite{LV16} of 2016, 
 which covered the history and the State of the Art of  that time.  The papers \cite{P20},  
 \cite{LPKZ20},  and \cite{P22}
 have reported new significant progress.  Moreover, \cite{P22} is a comprehensive study of black box polynomial root-finders, which
covers  subdivision and various functional iterations: Newton's, Schr{\"o}der's, Weierstrass's, aka Durand -- Kerner's, and Ehrlich's, aka Aberth's.}

(ii)  The black box  algorithms 
run fast for  polynomials that can be evaluated fast,
  e.g., are sparse, shifted sparse or Mandelbrot's, and can be readily
  extended  to efficient eigen-solving for matrices that can be inverted fast. 
  
  (iii) The algorithms allow numerically safe Taylor's shifts and scaling of the variable $x$ (see \ref{eqshft}), that is, can be applied to the polynomial $t(x)=p(\frac{x-c}{\rho})$ for any pair of complex $c$ and positive $\rho$.
This enables efficient refinement 
of approximate zeros of $p(x)$
by using our algorithms as well as    
$k$-fold  parallel acceleration 
of root-finding by using $k$ processors and
   concurrent application of our root-finders with shifts to $k$ distinct 
   centers $c_1,\dots,c_k$.
   
Clearly, Newton's ratio $-p(x)/'p(x)$
 vanishes at the zeros of $p(x)$. In Sec. \ref{sapplstrtf} we combine this simple observation with our algorithms of \cite[Secs. 6.2 and 6.3]{P22} for root radii of a black box polynomial to  
 compute non-costly  initialization
    of functional iterations such as Newton's, Ehrlich's (aka Aberth's), and Weierstrass's (aka Durand-Kerner's) for black box polynomial root-finding.
   
   {\bf 1.7. Further acceleration of root-finding: some recipes and  an algorithm.}
 (i)  We propose heuristic acceleration of our algorithms based on performing  fixed or random rotations of the unit disc, but so far (ii) an alternative black box polynomial
 root-finder in \cite{GS22}
 seems to accelerate computation of the $2^k$th power of the two extremal  zeros of $p(x)$
 by a factor of $2^k/k$ versus our algorithms based on power sum approximation.
  The paper \cite{GS22} extends DLG iterations for Newton's inverse ratios based on equation
 \begin{equation}\label{efggnwtrt}
\frac{p'_{h+1}(x)}{p_{h+1}(x)}=\frac{1}{2\sqrt x}\Big(\frac{p_h(\sqrt x)}{p_h(\sqrt{-x})}-\frac{p_h(\sqrt {-x})}{p'_h(\sqrt x)}\Big),~h=0,1,\dots,
 \end{equation}
 appeared as Eqn. (78) 
 in a revision of \cite{P22} of January   2022.  The paper  \cite{GS22} ingeniously performs these iterations by using polar  representation of complex 
 numbers $x$, in  which $\sqrt{x}$ 
and $\sqrt{-x}$ are  computed
 very fast. As in \cite{MZ01}, application of polar  representation enabled numerical stabilization, but  the resulting algorithm of \cite{GS22} is deterministic, fast,
   and can be applied to a black box polynomial $p(x)$.
   
    Authors' extensive tests in the Graduate Center of  the City University of New York have demonstrated the  efficiency 
     of that algorithm for approximation  
 of the extremal root radii of the test polynomials of MPSolve --  the package of polynomial root-finding subroutines, currently of user's choice. 

(iii) To extend  \cite{GS22} to approximation of the complex  extremal zeros of $p(x)$ themselves  rather than their high powers, one can apply the descending process of Sec. 1.3
or can try to extend the algorithm of \cite{GS22} to FG iterations for Newton's inverse ratios: 

  \begin{equation}\label{eqfgnwtrt}
\frac{q_{h+1}(x)}{p_{h+1}(x)}=\frac{1}{2\sqrt x}\Big(\frac{q_h(\sqrt x)}{p_h(\sqrt{-x})}-\frac{q_h(\sqrt {-x})}{p_h(\sqrt x)}\Big),~h=0,1,\dots
\end{equation}

(iv) Would application of polar representation be as  efficient in the case of original
DLG and FG iterations applied to 
a polynomial $p(x)$ rather than to 
the inverse Newton's ratio? 
Probably not so, but it is interesting to investigate this.
 
\medskip
 
{\bf 1.8. Organization of the paper.}  We devote the next section  to some background material.   In Sec. \ref{sapplstrtf} 
we extend our recent algorithnms for root-radii approximation to a new variant of Lehmer's root-finder and to fast initialization of functional iterations for polynomial root-finding. In Sec. \ref{sextrrts} 
 we reduce
approximation of zeros of a polynomial to  approximation of the power sums of its zeros. In Secs. \ref{scmphod} and \ref{spwrunt}  we approximate the power sums of the zeros of a polynomial
in two ways -- by using Newton's identities and by means of approximation of Cauchy integrals.  In Sec. \ref{sall}
we extend root-finding from the absolutely smallest and largest zero of $p(x)$ to  all its zeros.
   In Sec. \ref{sbnfts} 
we specify some benefits of black box polynomial root-finders. 
 In Sec. \ref{sestextrrtrd} 
we  estimate extremal root radii, partly
relaxing the root separation assumptions  (\ref{eqnndcrs}).  We  devote  short
Sec. \ref{scnc} to conclusions. 
   
 
\section{Background: Basic definitions, auxiliary results, and 
some applications to root-finding}\label{sbckg} 

\subsection{Definitions}\label{sdef}

 
 \begin{itemize}
 \item%
We write ``roots`` for 
``the roots of the equation  $p(x)=0$`` and enumerate them 
in  non-decreasing  order of their absolute values:
$|x_1|\ge |x_2|\ge\cdots\ge|x_d|$.
 \item
 $D(c,\rho)$, 
$C(c,\rho)$, and
$A(c,\rho,\rho')$ 
denote a disc, a circle (circumference),  and an
annulus on the complex plane,  
respectively:
$$D(c,\rho):=\{x:~|x-c|\le \rho|\},~C(c,\rho):=\{x:~|x-c|= \rho|\},$$
$$
~A(c,\rho,\rho'):=\{x:~\rho\le |x-c|\le \rho'|\}.$$
 
\item
  A disc $D=D(c, \rho)$  and its boundary circle $C=C(c, \rho)$ are $\theta$-isolated for $\theta\ge 1$
 if  the annulus $A(c,\rho/\theta, \theta\rho)$ contains no roots.
Supremum of such values 
 $\theta$ is the {\em  isolation} of
 $D$ and $C$.
 ``{\em  Isolated}`` stands for 
``$\theta$-isolated`` where $\theta-1$  
exceeds a positive constant.
\item
Write $\zeta$ and $\zeta_q$ for a primitive $q$th root of unity
     \begin{equation}\label{eqzt} 
\zeta:=\zeta_q:=
\exp\Big(\frac{2\pi \sqrt{-1}}{q}\Big).
\end{equation} 
\item
 For a complex $c$ and a positive $\rho$, {\em Taylor's shift} of the variable, or translation,  together with {\em scaling},
\begin{equation}\label{eqshft}
x\Longleftrightarrow  
y=\frac{x-c}{\rho},
\end{equation}   
map polynomials, discs and circles:
\begin{equation}\label{eqshftp}
p(x)\Longleftrightarrow t_{c,\rho}(y)=p\Big(\frac{x-c}{\rho}\Big),
\end{equation} 
$$ D(c,\rho):=\{x:~|x-c|\le\rho\}\Longleftrightarrow D(0,1),$$
  $$C(c,\rho):=\{x:~|x-c|=\rho\}\Longleftrightarrow C(0,1).$$
  \end{itemize}

\subsection{Auxiliary results}\label{saux}  
   
    Substitute factorization  (\ref{eqpoly}) into the ratio $-\frac{p'x)}{p(x)}$ and obtain the following well-known
    equation,
\begin{equation}\label{eqratio}
-\frac{p'(x)}{p(x)}=
\sum_{j=1}^d\frac{1}{x_j-x},
\end{equation} 
which  implies that
\begin{equation}\label{eqratio0r}
-\frac{p'(0)}{p(0)}=
\sum_{j=1}^d\frac{1}{x_j}~{\rm and}~
-\frac{p'_{\rm rev}(0)}{p_{\rm rev}(0)}=
\sum_{j=1}^dx_j.
\end{equation}

                                                                                                                                         We can approximate the value
$p'(c)$ for a black box polynomial $p(x)$ based on the equation $p'(c)=\lim_{x\rightarrow c}\frac{p'(x)}{p(x)}$ 
and can evaluate $p'(c)$
 by applying the algorithm that supports the following  result.
                                                                                                                                                                                                                                                                                
\begin{theorem}\label{thlnm}  
 Given  a black box  function $f(x)$
   over a field $\mathcal K$ of constants
 that has a derivative 
   and a straight line algorithm (that is, one with no branching) that evaluates $f(x)$ at a point $x$  by using $A$ additions and subtractions, $S$ scalar multiplications (that is, multiplications by elements of the field $\mathcal K$), and $M$ other multiplications and divisions, one can extend this algorithm to the evaluation at $x$ of both $f(x)$ and $f'(x)$ by using 
  $2A + M$ additions and subtractions, $2S$ scalar multiplications,
and $3M$ other multiplications and divisions.
\end{theorem}
\begin{proof}
See \cite{L76} or \cite[Thm. 2]{BS83} for a constructive proof of this theorem for any function $f(x_1,\dots,x_s)$ 
 that has partial derivatives in all its $s$ variables $x_1,\dots,x_s$.
\end{proof}

\subsection{Some known root radii estimates and algorithms}\label{srtrd}

Given  the coefficients of $p(x)$ one can readily compute the following narrow range for the extremal root radii $|x_d|$ and $|x_1|$ 
(cf. \cite[Sec. 6.4]{H74},
 \cite{MS99,Y00})  and  can  closely approximate all the $d$ root radii   at a low Boolean cost:
\begin{equation}\label{eqrtrdbnds}
\frac{1}{d}\tilde r_+\le |x_1|<2\tilde r_+,~\frac{1}{2}\tilde r_-\le |x_d|\le d~\tilde r_-
\end{equation} 
    where 
    \begin{equation}\label{eqrtrdbndsrev}
 r_-:= 
\min_{i\ge 1}\Big|\frac{p_0}{p_i}\Big|^{\frac{1}{i}}~{\rm and}~
 r_+:=\max_{i\ge 1} \Big|\frac{p_{d-i}}{p_d}\Big|^{\frac{1}{i}}.
\end{equation}
     

In particular, for $i=1$ Eqns.
(\ref{eqrtrdbnds}) and  (\ref{eqrtrdbndsrev}) together
imply that
\begin{equation}\label{eqrtrdbnd}
r_d\le d\Big|\frac{p_0}{p_1}\Big|=d\Big|\frac{p(0)}{p'(0)}\Big|.
\end{equation}

  
\begin{theorem}\label{thrr}  \cite[Prop. 3]{IP21}.\footnote{A variation of this theorem was  
first outlined and briefly analyzed in \cite{S82} for fast approximation of a single root radius and then extended
in \cite[Sec. 4]{P87} (cf. also \cite[Sec. 5]{P89})  to approximation of all root radii of a polynomial.}   
 Given the coefficients of a polynomial $p(x)$, one can approximate all the
 $d$ {\em root radii} $|x_1|,\dots,|x_d|$  within the relative error bound 
$4d$ at a Boolean cost in $O(d\log(||p||_{\infty}))$. 
\end{theorem} 

\cite{IP21} extends the supporting algorithm for this theorem to approximation of all root radii within a relative error bound $\Delta$
by performing   
$\lceil\log_2\log_{1+\Delta}(2d)\rceil$ DLG root-squaring iterations 
(\ref{eqdnd}).
Combine 
Thm. \ref{thrr} with Boolean cost estimates for these iterations, implicit in the proof of \cite[Cor. 14.3]{S82}, and obtain

 
\begin{corollary}\label{corr}   
 Given the coefficients of a polynomial $p=p(x)$ and a positive $\Delta$, one can approximate all the
 $d$ {\em root radii} $|x_1|,\dots,|x_d|$  within a relative error bound 
$\Delta$ in $\frac{1}{d^{O(1)}}$ at a Boolean cost in $O(d\log(||p||_{\infty})+d^2\log^2(d))$. 
\end{corollary} 
The algorithm  supporting the corollary  
  computes at a low cost $d$ concentric narrow annuli covering the circles
$C(0,|x|_{j})$  for $j=1,\dots,d$
(some annuli may pairwise 
overlap),  whose union $\mathbb U$ contains all $d$ roots. 
 \cite{IP21} strengthens the benefits of having such annuli  by also computing the $d$ root radii from a fixed positive $c$ and  
$c~\sqrt{-1}$. In particular 
the search of the roots, that is, the zeros of $p(x)$, is reduced to the intersection of the three 
unions.

All these algorithms and estimates 
for root radii, and even for the extremal root radii, only apply to a polynomial $p(x)$ given with its coefficients, but the randomized algorithms of \cite[Sec. 6.2]{P22}  fast approximate the $j$th  smallest root radius for any $j$  of a black box polynomial (slightly faster for $j=1$ and $j=d$). \cite[Sec. 6.3]{P22}  extends these algorithms to approximation of all $d$ root radii of a black box polynomial at the price of slow down only by a factor of $\log(d)$; even for all root radii the algorithm of \cite[Sec. 6.3]{P22} evaluates the ratio 
$\frac{p'(x)}{p(x)}$  at the 
number of points $x$  nearly linear  
 in $d$.
 
\begin{theorem}\label{thrrcst}
Suppose that a black box polynomial $p(x)$ of a degree $d$ has precisely $m$ zeros, counted with their multiplicities, that lie in isolated unit disc $D(0,1)$. Then for fixed positive   $\phi$, $\epsilon=1/2^{b}$, and $v>1$, 
the randomized algorithm of  \cite[Secs. 6.2 and 6.3]{P22}
  performs with a probability of error at most
 $1/2^v$, 
evaluates the ratio $p'(x)/p(x)$ at $O(v m\log (b/\phi))$ points, and
for any fixed $j$, $1\le j\le m$,  approximates the $j$-th smallest
root radius $|x_j|$ within the relative error $2^{\phi}$
or determines that  $|x_j|\le \epsilon$. At the price of increasing the number of evaluation points by a factor of $O(\log(m))$ the algorithm 
outputs such approximations to all
$m$ root radii.
\end{theorem}

 This black box algorithm  can
 be readily extended to approximate  the distances from
any fixed complex center $c$ to all zeros of a polynomial lying in a  disc $D(c,\rho)$ for any fixed positive $\rho$ as well as the
distances from a fixed center $c$ to all zeros of $p(x)$.
This implies  further benefits -- see \cite{P17} and \cite[Secs. 10.6 and 10.7 ]{P22}.



\section{Some applications of root radii computation to black box polynomial root-finding}\label{sapplstrtf}

Recall  Lehmer's celebrated polynomial root-finder and devise its new  black box variant  based on our black box algorithms for extremal root radii and on the following two observations: 
(i) Newton's ratio $-p(x)/'p(x)$ vanishes at the zeros of $p(x)$ and 
(ii) the circle $C(0,r_d)$ contains such  zeros or a zero. 
    
 \begin{algorithm}\label{algrtrdnwt} {\em [Lehmer-Newton's root-finder.]}   
 \begin{description} 
 INPUT :  a black box polynomial $p(x)$ of a degree $d$
 and a positive tolerance value 
 $\epsilon$.

  OUTPUT: a complex $z$, approximating  a zero of $p(x)$ within $\epsilon$.
      
  COMPUTATIONS:   
  \begin{enumerate}
 \item 
 Compute an  upper bound 
 $r_d'=d|\frac{p(0)}{p'(0)}|$
   on the smallest root radius 
  $r_d=|x_d|$ and output $z=0$ if $r_d'\le\epsilon$.
\item
Otherwise compute a refined  upper    bound $r_d''$ on $r_d$ by applying a root radius
algorithm of \cite[Sec 6.2]{P22}.
Output $z=0$ if $r_d''\le\epsilon$.

\item
Otherwise compute   the values of
$-\frac{p(x)}{p'(x)}$ (Newton's ratios)
at $q$ equally space points of the circle $C(0,r'_d)$ 
 for a sufficiently large $q$.
 \item  
 Choose a point $c$  
at which the value of the ratio  is absolutely smallest
If $d|\frac{p(c)}{p'(c)}|\le \epsilon$, than output $z=c$.
 \item 
Otherwise shift the origin into this point and
go to stage 2.
 \end{enumerate}
 \end{description} 
 \end{algorithm}   
 \begin{remark}\label{relehnt}
 Between stages 3 and 4, one can apply a root-refiner at $c$, e.g., Newton's iterations.
\end{remark} 
 
If approximations $r_{j}'$ have been computed to $w$ root radii $r_j$,  
 for $j=1,\dots,w$,
 then one can
concurrently  apply stage 2 above to the $d$ circles $C(0,r'_j)$, for $j=1,\dots,d$, and then can use the $w$ computed approximations to the $w$ zeros of $p(x)$ to initialize various functional iterations
for root-finding, e.g.,  Newton's or Schr{\"o}der's. For $w=d$ there are more option of functional iterations, e.g., Ehrlich's  (aka Aberth's) or Weierstrass's (aka Durand-Kerner's) (see \cite{M07} and \cite{MP13} for a variety of such iterations).


\section{Extremal roots from the power sums of the roots}\label{sextrrts}  

\begin{theorem}\label{thvtprsm}
[See Remark \ref{redlgext}.]
Let $f_k(z):=f\cdot \prod_{j=1}^{w}(z-z_j^{k})$ and $\widehat f_k(z):=\widehat f\cdot \prod_{j=1}^{w}(z-z_j^{-k})$,
for two nonzero scalars $f$ and $\widehat f$, so that  $\widehat  f(z)$ is a scaled reverse polynomial of $f(z)$.
Then
\begin{equation}\label{eqvtpwrsms} 
-\frac{f_k'(0)}{f_k(0)}=\sum_{j=1}^{w}z_j^{-k}~{\rm
and}~-\frac{\widehat f_k'(0)}{\widehat f_k(0)}=\sum_{j=1}^{w}z_j^{k}~{\rm
for}~k=0,1,\dots.
\end{equation}
 \end{theorem}
\begin{proof}
 Substitute $p(x):=f_k(z)$ into Eqn. (\ref{eqratio0r}).
\end{proof}

\begin{remark}\label{redlgext}
Thm. \ref{thvtprsm} does not involve
DLG iterations, but  
 let $\widehat f_k(z):=p_{h,{\rm rev}}(x)$
 denote the reverse polynomial of
 $p_h(x)$,  write 
\begin{equation}\label{eqdlgspwrsms0}
z:=x,~k=2^h,~{\rm and}~f_k(z):=p_h(x),
\end{equation}
apply Thm. \ref{thvtprsm}, and obtain
\begin{equation}\label{eqdlgspwrsms}
-\frac{p_{d-1}^{(h)}}{p_{d}^{(h)}}=-\frac{p'_h(0)}{p_h(0)}=\sum_{j=1}^dx_j^{-k},~-\frac{p_{1}^{(h)}}{p_{0}^{(h)}}=-\frac{p'_{h,{\rm rev}}(0)}{p_{h,{\rm rev}}(0)}=\sum_{j=1}^dx_j^{k}~{\rm for}~k=2^h.
\end{equation} 
\end{remark}

Now  reduce approximation
of the absolutely smallest (resp. largest) zero of a polynomial to  approximation 
of two consecutive sufficiently large power sums of its zeros (resp. reciprocals of its zeros) and bound the approximation errors. Such a bound
is proportional to the $k$th power of the ratio $|\frac{z_w}{z_{w-m}}|$ or
$|\frac{z_{m+1}}{z_{1}}|$ and thus decreases fast as the ratio decreases. The ratio can be computed at a low cost by means of the fast algorithms that support Cor. \ref{corr}  and  Thm \ref{thrrcst} and {\em is small if
we begin with an approximation
to a zero of $p(x)$ and apply our algorithms to refine it fast.}

\begin{corollary}\label{covtprsm}
For polynomials $f_i(z)$ and $\widehat f_i(z)$ of Thm. \ref{thvtprsm}, $i=k,k+1$, and two integers $k\ge 0$ and $m$, $1\le m\le w$,  it holds that \\
(i) $-\frac{f_k'(0)}{f_k(0)}= 
(1+\Delta_{w,k,m})mz_{w}^k$ 
{\rm and} $\frac{f_{k+1}'(0)f_{k}(0)}{f_{k+1}(0)f'{k}(0)}= 
(1+\gamma_{w,k,m})z_{w}$  {\rm where} 
\begin{equation}\label{eqgamma_ell}
1+\gamma_{w,k,m}=\frac{1+\Delta_{w,k+1,m}}{1+\Delta_{w,k,m}},~ 
|\Delta_{w,i,m}| \le\frac{w-m}{m}\Big|\frac{z_{w}}{z_{w-m}}\Big|^i,~i=k,k+1,
 \end{equation}
{\rm  if}
\begin{equation}\label{eqnndcrsd} 
|z_{w}|=|z_{w-1}|=\cdots=|z_{w-m+1}|<|z_{w-m}|=\min_{j\ge w-m}|z_j|. 
 \end{equation}
 
 (ii)  $-\frac{\widehat f_k'(0)}{\widehat f_k(0)}=(1+\Delta_{1,k,m})m 
z_1^k$ {\rm and} 
$\frac{\widehat f_{k+1}'(0)\widehat f_{k}(0)}{\widehat f_{k+1}(0)\widehat f'_{k}(0)}= 
 (1+\gamma_{1,k,m})z_1$ {\rm where}
\begin{equation}\label{eqgamma_1}
1+\gamma_{1,k,m}=\frac{1+\Delta_{1,k+1,m}}{1+\Delta_{1,k,m}},~
|\Delta_{1,i,m}| \le\frac{w-m}{m}\Big|\frac{z_{m+1}}{z_{1}}\Big|^i,~i=k,k+1
 \end{equation}
 {\rm if} 
\begin{equation}\label{eqnndcrsdmx} 
|z_1|=|z_2|=\cdots=|z_m|>|z_{m+1}|=\max_{j>m}|z_j|.
 \end{equation}
\end{corollary}  
  
\begin{corollary}\label{cogammab}
 Under the assumptions of Cor. \ref{covtprsm} \\
 (i) let (\ref{eqnndcrsd}) hold and let $\Delta_{v,m}\ge\Big|\frac{z_{w}}{z_{w-m}}\Big|$.
Then  
\begin{equation}\label{eqgammav}
\gamma_{w,k,m}\le \frac{2\Delta_{w,m}^k}{1-\Delta_{w,m}^k},
\end{equation}
\begin{equation}\label{eqgammavb}  
   \gamma_{w,k,m}\le\epsilon=1/2^b~ {\rm for}~k\ge (b+2)
\log_{\Delta_{v,m}}(2)~{\rm if}~ 2\Delta_{v,m}^k\le 1.
 \end{equation}
 Likewise, (ii) let (\ref{eqnndcrsdmx}) hold and let $\Delta_{1,m}\ge \Big|\frac{z_{m+1}}{z_{1}}\Big|$.
Then  
\begin{equation}\label{eqgamma1}
\gamma_{1,k,m}\le \frac{2\Delta_{1,m}^k}{1-\Delta_{1,m}^k},
\end{equation}
\begin{equation}\label{eqgamma1b}  
   \gamma_{1,k,m}\le\epsilon=1/2^b~ {\rm for}~k\ge (b+2)
\log_{\Delta_{1,m}}(2)~{\rm if}~ 2\Delta_{1,m}^k\le 1.
 \end{equation}
\end{corollary}
\begin{proof}
First deduce bound (\ref{eqgammav})
 from (\ref{eqgamma_ell}); from this obtain  that
$\gamma_{w,k,m}\le 4\Delta_{v,m}^k$  
if $2\Delta_{v,m}^k\le 1$.
Now prove claim (i) by
taking binary logarithms on both sides.
Similarly prove claim (ii).
\end{proof}

\section{Approximation of the power sums of polynomial zeros}\label{scmphod}


\subsection{An overall error bound  for approximation of an extremal zero of a black box polynomial}\label{serrextr} 
 

Combine (\ref{equ7.12bdrh}) 
with bounds of Cor. \ref{cogammab}, 
for simplicity assuming that $f(x)=p(x)$ and $v=d$ (extension to  $f(x)$ being a factor of $p(x)$ is a straightforward exercise), and  obtain 
 
\begin{theorem}\label{therrextr}
Given two tolerance values 
$\epsilon=1/2^{b}$ and
$\epsilon_0=1/2^{b_0}$
and the 
bounds $\Delta_{d,m}$ and  
$\Delta_{1,m}$ of Cor. \ref{cogammab} (see Sec. \ref{srtrd}).
Choose integer $k$ to ensure an upper bound $\epsilon=1/2^b$ on the approximation error of $x_w$ and $x_1$ by the power sums of the zeros of $p(x)$ according to Cor. \ref{cogammab}. Compute approximations $s_k'$ and $s'_{k+1}$,
  to the power sums $s_k$ and 
$s_{k+1}$ of the zeros of $p(x)$
as well as approximations $\bar s'_k$ and  $\bar s'_{k+1}$ 
  to the power sums $\bar s_k$ and 
$\bar s_{k+1}$ of the zeros of the reverse polynomial $p_{\rm rev}(x)$, respectively, in all cases
within a  fixed tolerance bound $\epsilon_0$. Then
$$|\frac{s_{k+1}'}{s_{k}'}-x_d|\le\epsilon+\epsilon_0~{\rm and}~|\frac{\bar s_{k+1}'
}{\bar s_{k}}-x_1|\le\epsilon+\epsilon_0.$$
\end{theorem}

\subsection{Newton's and Cauchy's approximation of the power sums}\label{scmpalgs}
Next we  approximate 
the power sums of the zeros and the reciprocals of  the zeros  of
 a polynomial by expressing them first via 
Newton's identities and then via Cauchy integrals.   Newton's identities only enable us to compute the power sums for $p(x)$ itself and involve a set of its leading or trailing coefficients. With Cauchy integrals we scale the variable  and increase arithmetic cost to avoid numerical stability problems but express the power sums for a black box polynomial $p(x)$ as well as for its factors via the values of Newton's inverse ratio $-\frac{p'(x)}{p(x)}$ for a black box polynomial $p(x)$,
leading to significant computational benefits.

\subsection{Newton's Power Sums Computation}\label{spwrntn}

Given the  $k+2$ trailing coefficients $p_0=1,p_1,\dots,p_{k+1}$ of a  polynomial  $p(x)$, with  $p(0)=1$, which are the
$k+2$ leading coefficients $p_d',p_{d-1}',\dots,p_{d-k-1}'$
of the monic reverse polynomial
$p_{\rm rev}(x)$, we can compute  
the $k+1$ power sums $s'_1,\dots,s_{k+1}'$ of $p_{\rm rev}(x)$ by solving the triangular  Toeplitz linear system  of $k+1$ Newton's identities, 
  which is equivalent to computing
the reciprocal of $p_{\rm rev}(x)$ modulo $x^{k+2}$ \cite[Sec. 2.5]{P01}:
    
~~~~~~~~~~~~~~~~~~~~~~~~~~~~~~~~ $s_1'+p'_{d-1}=0$,
  
~~~~~~~~~~~~~~~~~~~~~~~~ 
$p'_{d-1}s_1'+s'_2+2p'_{d-2}=0$,
 
~~~~~~~~~~~~~~~~~~~ $p'_{d-2}s_1'+p_{d-1}'s_2'+s_3'+3p_{d-3}'=0$,
 
~~~~~~~~~~~~~~~~~~~ $\dots\dots\dots\dots\dots\dots\dots\dots\dots\dots\dots$ 
  
\begin{equation}\label{eqntnidps}    
s_i'+\sum_{j=1}
^{i-1}p'_{d-j}s'_{i-j}=-ip'_{d-i},
~i=1,\dots,k+1.
 \end{equation} 
 We refer to this classic algorithm as  {\em  Newton's Power 
 Sums}.
   
{\bf Arithmetic complexity.}  We can solve linear system (\ref{eqntnidps}) at  arithmetic cost of $2(k+1)^2$ by using back substitution
   but can decrease it to $O(k\log(k))$ with FFT  \cite[Sec. 2.6]{P01}.
  
   Given $k+2$ leading coefficients of
   $p(x)$ we can similarly compute the $k+1$ power sums $s_1,\dots,s_{k+1}$ of
   $p(x)$.
   
\subsection{Cauchy integral  algorithms}\label{spwrchint}

  {\bf 1. Integral formula.} For a non-negative integer $h$, express the $h$th power sum of the $w$ roots of a polynomial $p(x)$ lying in a domain $\mathcal D$ on the complex  plane
    as a Cauchy integral over a  boundary  $\mathcal C$ of this domain:
\begin{equation}\label{eqint}
s_h:=\sum_{j=1}^wx_j^{h}=\frac{1}{2\pi \sqrt{-1}}\int_{\mathcal C}\frac{p'(x)}{p(x)}~x^h ~dx.
\end{equation} 
  
$s_h$ is the power sum of all $d$  zero of $p(x)$ if all of them lie in
the domain $\mathcal D$. 

To express the power sums of the reciprocals of the zeros of  $p(x)$,
substitute $p_{\rm rev}(x)$ for $p(x)$
into (\ref{eqint}).

{\bf 2. Algorithmic options.} 
We must assume some isolation of 
the boundary contour $\mathcal C$ from the zeros of  $p(x)$ for otherwise the output errors of the 
 known algorithms   for  approximation of the integral (\ref{eqint}) are unbounded. 
Next we recall some algorithmic options.

(i) It is attractive to approximate integral (\ref{eqint}) by applying trapezoid rule (see its exponentially convergent adaptive version  in the BOOST library) modified for arbitrary 
precision,\footnote{This was suggested to us by Oren Bassik
of the PhD program in Mathematics of the Graduate Center of CUNY.} although this does not support reduction modulo $x^q-1$,
which would simplify computations in the case where $q\ll d$ and a polynomial $p(x)$ is given with its coefficients. 

(ii) Kirrinnis in \cite{K00} approximates Cauchy integrals over various smooth contours $\mathcal C$.

(iii) By following  \cite{S82} choose  $\mathcal C=C(0,1)$ and  
 approximate integral (\ref{eqint}) 
 with finite Cauchy sums $s_{h,q}$ defined as  follows.

 \begin{definition}\label{defnps}
For a polynomial $p(x)$, two  integers $q$ and $h$ such that $0\le h<q$, and $\zeta_q:=
\exp\Big(\frac{2\pi \sqrt{-1}}{q}\Big)$
of  (\ref{eqzt}), define  {\rm Cauchy sum}  
  \begin{equation}\label{equ7.12.2}
s_{h,q}: =\frac{1}{q}\sum_{g=0}^{q-1}\zeta_q^{(h+1)g}~\frac{p'(\zeta_q^g)}{p(\zeta_q^g)}.
\end{equation}
\end{definition}
 
We proved in  \cite{P22} and  recalled in \cite{IP20} the following result.
 
\begin{theorem}\label{thpwrsm0}
For a polynomial $p(x)$ of (\ref{eqpoly})
and a positive integer $q$,  the Cauchy sum $s_{h,q}$ of (\ref{equ7.12.2}) satisfies
$s_{h,q}=\sum_{j=1}^d\frac{x_j^h}{1-x_j^q}$ for $h=0,1,\dots,q-1$ unless $x_j^q=1$ for some integer $j$, $1\le j\le d$.
\end{theorem}

 \cite{P22}  extensively  applies
  Cauchy sums $s_(h,q)$ to root-finding  in the case of small integers $h$l next we recall some relevant results
  and algorithms.
  
     \section{Cauchy approximation of the power sums of the zeros of $p(x)$ lying in the unit disc $D(0,1)$}\label{spwrunt}
  
  \subsection{Computation of Cauchy sums}\label{scmptcch} 
 
\begin{algorithm}\label{algcmppwrsm} {\em [Cauchy sum computation.]}   
 \begin{description} 
 INPUT :  $p(x)$, 
 $h$  and $q$ of Def. \ref{defnps} such that  
  $\prod_{g=0}^{q-1}p(\zeta_q)^g\neq 0$.

  OUTPUT:  the  Cauchy sum 
 $s_{h,q}$ of (\ref{equ7.12.2}).
      
  COMPUTATIONS:   
  \begin{enumerate}
 \item
Compute the values $r_g:=\frac{\bar p'(\zeta^g)}{\bar p(\zeta^g)}~{\rm for}~g=0,1,\dots,q-1.$
  \item
 Compute and output 
  $ s_{h,q}:=\frac{1}{q}\sum_{g=0}^{q-1}\zeta^{(h+1)g}r_g$.
 \end{enumerate}
 \end{description} 
 \end{algorithm}


{\bf 1. Arithmetic cost.} Alg. \ref{algcmppwrsm} evaluates 
  the ratio $\frac{p'(x)}{p(x)}$
 at $q$ points at stage 1 and performs $2q$ arithmetic operations at stage 2.

{\em Given the  coefficients 
 of  $p(x)$ and $q<d$
we  perform stage 1 faster:} (i) reduce the polynomials $p(x)$
 and  $p'(x)$ modulo $x^q-1$  at the cost of performing $2d-2q+1$ subtractions,
  (ii) evaluate the two resulting polynomials at the $q$th roots of unity,  that is, perform  discrete Fourier transform (DFT) by using $O(q\log(q))$  arithmetic operations, and (iii) compute the $q$ ratios of the $q$ pairs of the values output by the two DFTs at the $q$th roots of unity.
 \medskip
 
{\bf 2. Approximation errors.} 

 We can readily deduce from  Thm. \ref{thpwrsm0} the following estimates 
 of \cite{S82}.

\begin{theorem}\label{thpwrsm}
For a polynomial $p(x)$, two integers $h$ and $q$, $0\le h<q$, and  $\theta>1$, let  the annulus $\{x:~\frac{1}{\theta}\le |x|\le \theta\}$ contain no zeros of $p(x)$.
 Then 
\begin{equation}\label{equ7.12.6} 
|s_{h,q}-s_h|\le
\frac{d\theta^{h}}{\theta^{q}-1}~{\rm for}~h=0,1,\dots,q-1,
\end{equation} 
and so
\begin{equation}\label{equ7.12eps}  
  |s_{h,q}-s_h|\le\epsilon_0=1/2^{b_0}
  \end{equation}
   for a fixed positive $\epsilon_0$ if 
  \begin{equation}\label{equ7.12} 
q-h\ge \log_{\theta}\Big(1+\frac{d}{\epsilon_0}\Big)=\log_{\theta}(1+b_0d)~{\rm for}~h=0,1,\dots,q-1.  
  \end{equation}
\end{theorem}
 
\begin{remark}\label{retcnt}
(\ref{equ7.12.6}) is an upper bound
on the value $|s_{h,q}-s_h|$; it can be pessimistic. For a heuristic recipe
towards its decrease we can rotate the disc $D(0,1)$ by fixed or random angles and try to deduce from Thm. 
\ref{thpwrsm0} and/or test empirically whether this can help decrease approximation error.
\end{remark}


\subsection{Cauchy sums in  any disc}\label{scchanyd}  
 

Taylor's shift with scaling   enables us to 
reduce the computation of Cauchy sums $s_{h,q}$ in any disc $D(c,\rho)$ to the case of the unit disc $D(0,1)$.

\begin{definition}\label{defcchy} 
Apply Taylor's shifts  with scaling  (\ref{eqshft}) to   
define the {\em Cauchy sums for a polynomial $p$,  a positive integer $q$, and the disc} $D(c,\rho)$ as follows:
\begin{equation}\label{equ7.12.00}
s_{h,q}(p,c,\rho): =\frac{\rho}{q}\sum_{g=0}^{q-1}\zeta^{(h+1)g}~\frac{p'(c+\rho\zeta^g)}{p(c+\rho\zeta^g)}~{\rm for}~\zeta~{\rm of~(\ref{eqzt})~and}~h=0,1,\dots,q-1, 
\end{equation}   
that is, $s_{0,q}(p,c,\rho)$ is the Cauchy sums $s_{0,q}(t,0,1)$ for the positive integer $q$, the polynomial $t(y)=p(\frac{y-c}{\rho})$,  and the unit disc $D(0,1)$.
\end{definition} 
Clearly, Taylor's shift with scaling 
(\ref{eqshft}) preserves isolation of a disc but changes
the  zeros of $p(x)$, their power sums, and Cauchy sums. One can, however, apply transformation $x\Longrightarrow y$ of (\ref{eqshft}) to map a disc 
$D(c,\rho)$ into the unit disc $D(0,1)$,
approximate the zeros $y_j$ of $t(y)$
lying in that disc, and then apply 
 converse map $y\Longrightarrow x=c+\rho  y$ to 
approximate the zeros $x_j$  of
 $p(x)$.

 
 To  estimate the errors of the  approximation of the power sums (\ref{eqint}) where $\mathcal D=D(c,\rho)$, combine Thm. \ref{thpwrsm}
with equations $x_j=c+\rho  y_j$ where
$y_j$ denote the zeros of $t(y)=p(\frac{y-c}{\rho})$ lying in the unit disc $D(0,1)$. 
  
  For discs $D(0,\rho)$, centered at the origin, extension of bound (\ref{equ7.12})
  is particularly   simple: $|s_{h,q}-s_h|\le \epsilon_0=1/2^{b_0}$ if
\begin{equation}\label{equ7.12rh} 
q-h\ge \log_{\theta}\Big(1+\frac{d \rho}{\epsilon_0}\Big)~{\rm for}~h=0,1,\dots,q-1,  \end{equation}
 and then we immediately extend  bound (\ref{equ7.12}) by just replacing $b_0d$ with $\rho b_0d$, that is,
 the 
  bound $|s_{h,q}-s_h|<1/2^{b_0}$ of (\ref{equ7.12eps}) is ensured if
\begin{equation}\label{equ7.12bdrh}
 q-h\ge \log_2(\rho b_0d).
 \end{equation}

\subsection{Numerical stability problems and a remedy}\label{sstbrm}
In the case where the unit disc is isolated and  an integer $h$ is large,
computation of  Cauchy sums $s_{h,q}$ is prone to  numerical stability problems,
 aggravated as $h$ grows larger. Indeed, all the zeros of $p(x)$
lying in the isolated disc $D(0,1)$ are noticeably exceeded  by 1. Therefore,
the Cauchy sums and the ratios
$|s_{h,q}|/\max_g \Big|\frac{p'(\zeta_q^g)}{p(\zeta_q^g)}\Big|$ converge to 0 exponentially fast as $h$ grows. Already for moderately large integers $h$,
 the terms 
$\frac{p'(\zeta_q^g)}{p(\zeta_q^g)}$  are nearly annihilated in the  computation of $s_{h,q}$ according to (\ref{equ7.12.2}), and this means numerical stability problems.

{\bf A remedy by means of scaling the variable.} Given an exponent $h$ and target error tolerance $\epsilon_0=1/2^{b_0}$, scale the variable $x\leftarrow \rho y$ and choose  
isolation $\theta>1$ such that
$\theta^h=2$, say. Then choose 
\begin{equation}\label{eqthbh}
\theta^h=2~{\rm and}~ q\ge h\log_2(1+d2^{b_0+1}).
   \end{equation}
  Substitute equation $\theta^h=2$ into (\ref{equ7.12.6})
and deduce that 
$|s_{h,q}-s_h|\le \frac{2d}{\theta^q-1}=\frac{2d}{2^{q/h}-1}\le 1/2^{b_0}=\epsilon_0$.

This bound on $q$ tends to be fairly large unless $h$ and/or $b_0$ are nicely bounded, which can be the case 
if, say, the value $|z_w/z_{w-m}|$
in Cor. \ref{cogammab} is small.
 

\section{Approximation of a sequence of the zeros of $p(x)$ with implicit deflation}\label{sall}

Let $g_0(x)$ denote  a polynomial $g(x)$ whose zero set 
$\{x_j,~j=1,2,\dots,w\}$ is the set of the zeros of $p(x)$ lying in the unit disc $D(0,1)$; in particular,  $g_0(x)=p(x)$  if $w=d$.

Suppose that we  have computed, e.g.,  by applying the algorithms of the previous sections,  an approximation  
 $z_1$ to a zero $x_d$ of $g_0(x)$  
in $D(0,1)$ 
such that $|z_1-x_d|\le \epsilon|z_1|$  for a fixed tolerance $\epsilon<1/4$,
 say.\footnote{Quite typically, we can very fast refine approximation of  the zeros $x_1,x_2,\dots$ by means of Newton's or Schr{\"o}der's iterations; alternatively, we can apply  the algorithms of this paper.}

We can deflate the factor
$x- z_1$ by approximating the quotient $ p(x)/(x-z_1)$ with a $(d-1)$st  degree polynomial and apply to it the same 
root-finder.  
 Recursively we can approximate all $w$ zeros of $p(x)$ lying in the disc $D(0,1)$.
 Such deflation destroys sparsity of $p(x)$ and blows up its overall  coefficient length (e.g., consider $p(x)=x^d+1$), but we can avoids these problems  by applying 
 implicit deflation where we only compute the values of the ratios
  \begin{equation}\label{eqimpldfl}
\frac{f'(x)}{f(x)}~{\rm for}~f(x)=\frac{p(x)}{g(x)},~
g(x)=\prod_{j=1}^k(x-z_j),
\end{equation}
for $z_j$  denoting the computed approximations to  $x_{d-j+1}$, $j=1,\dots,m$, but never compute polynomial coefficients.
 
 Dealing with black box polynomial
 root-finders we can reduce computation of Newton's inverse ratio
 $\frac{f'(x)}{f(x)}$ for $f(x)$ at some set of points $x$ to
the same task for $p(x)$
because of the following simple observation.
 
\begin{theorem}\label{thimpldfl}
Under (\ref{eqimpldfl}) it holds that
$$\frac{f'(x)}{f(x)}=
\frac{p'(x)}{p(x)}-
\sum_{j=1}^k\frac{1}{x-z_j}.$$
\end{theorem} 

If we can reuse the same points $x$ for approximation of the zeros $z_j$
 for all $j$, e.g., if we evaluate    
   Newton's inverse ratio at the $q$th roots of unity for a fixed integer $q$, we would only  need to
 recursively subtract from the values 
 $\frac{p'(x)}{p(x)}$ the reciprocals   
 $\frac{1}{x-z_j}$, each time performing just a single division (by $\frac{1}{x-z_j}$) 
and two subtractions.

In the particular case where a root-finder is reduced to power sum computation based on Thm \ref{therrextr}, we can merely
subtract the powers of computed roots from the computed power sums,
thus performing a single subtraction instead of two subtractions and division.

This would work if we approximate power sums with Cauchy sums where pairwise ratios of root radii are never large, but would not work if these ratios vary much because in that case we should apply scaling to avoid numerical stability problems and cannot reuse evaluation points.

Based on recursive application of Thm. \ref{therrextr}
we can output successive approximations to the $j$th absolutely 
smallest or absolutely largest  zeros of $p(x)$  for $j=1,2,\dots$, but  can similarly peel off recursively 
 the zeros 
closest to a fixed complex point $c$
by applying implicit deflation 
 to the polynomial $t_c(x)=p(x-c)$.
 
 We can alternatively apply this recipe  concurrently to reverse  polynomials of $t_{c_h}(x)=p(x-c_h)$ 
for $h=1,2,\dots$ to approximate a set of zeros of $p(x)$ closest to.
the points $c_1,c_1,\dots$ 

It is not simple to estimate the  integer parameters $h$ and $q$ at all  steps of recursive deflation a priori, but
we can estimate them {\em by action} provided that their growth
does not require to exceed  fixed tolerance bounds on the   
computational precision and arithmetic  cost.

\section{Some benefits of application of black box  polynomial root-finders}\label{sbnfts}

Black box  polynomial root-finders avoid numerical
stability problems caused by Taylor's shifts and scaling, which are basic tools of various popular root-finders, e.g., subdivision  root-finders.
This implies first two important benefits listed below. The algorithms 
have further benefits because they amount to the evaluation of Newton's inverse ratio  $-\frac{p'(x)}{p(x)}$
at sufficiently many points.

\begin{enumerate}
\item
{\em Relaxation of assumption (\ref{eqnndcrsd})}.  
 Assumptions (\ref{eqnndcrsd}) and even 
 (\ref{eqnndcrs})
hold with probability 1 for a polynomial $t_{c,\rho}(x)$ replacing $p(x)$
if all zeros $x_j$ of $p(x)$ are distinct
and if $c$ is a random variable sampled from any fixed disc $D$,
say, $D=D(0,1)$, under the uniform or standard Gaussian normal  probability
distribution on that disc.
\item 
{\em Concurrent root-finding on $m$ processors} by means of simultaneously
computing the DLG or FG sequences  
for the polynomials $t_{c_i,\rho_i}(x)$ for fixed or random sets  $c_i$
and $\rho_i$, $i=1,\dots,m$.
\item
Extension of approximation of a single zero
 $x_d$ to {\em recursive approximation
of $k$ absolutely largest or smallest zeros}
 for a fixed $k>1$ (that is, of all zeros for $k=d$) by means of implicit deflation (see Sec. \ref{sall}).
\item
{\em Acceleration} where $p(x)$ can be evaluated fast, e.g., is a sparse, shifted sparse, or Mandelbrot's polynomial, and
\item
{\em Numerically stable evaluation} of   polynomial 
$p(x)$  given in Bernstein's or Chebyshev's bases \cite[Sec. 1.2]{P22}. 
 \end{enumerate}

Our black box algorithms can be readily extended to {\em eigen-solving for a matrix or a polynomial matrix}, and this is efficient where a matrix or a polynomial matrix can be 
inverted fast, e.g., 
 is quasiseparable 
or  has small displacement rank.
This extension relies on the equations
$\frac{t'(x)}{t(x)}={\rm trace}((xI-T)^{-1})$  
where $T$ is a matrix with characteristic polynomial $t(x)=\det(xI-T)$ (implied by Eqn. (\ref{eqratio}))
and  
$\frac{t'(x)}{t(x)}={\rm  trace}(T^{-1}(x)T'(x))$  
where $T(x)$ is a matrix polynomial and $t(x)=\det(T(x))$ \cite[Eqn. (5)]{BN13}.  
Thus our algorithms approximate the eigenvalues of the matrix $T$ or matrix polynomial $T(x)$, which are  the zeros of the  characteristic polynomial of $T$ or $T(x)$, without computing the coefficients
of that polynomial. 
\medskip

\section{Estimation of extremal root radii}\label{sestextrrtrd}
 
Suppose that we do not apply the known algorithms and just rely on
 the following
well-known bounds on the extremal root radii:

 \begin{equation}\label{eqrtrdbndsrev1} 
|x_d|\le d~\Big |\frac{p(0)}{p'(0)}\Big |~{\rm and}~|x_1|\ge \Big |\frac{p'_{\rm rev}(0)}{d~p_{\rm rev}(0)}\Big|.
\end{equation}

By extending these bounds to
the polynomial $p_k(x)$ of Eqn. (\ref{eqdnd}) we obtain that

 \begin{equation}\label{eqratio0} 
 |x_d|^{2^k}\le d/ 
\Big| \Big(\frac{p_k'(0)}{p_k(0)}\Big)\Big|~{\rm and}~|x_1|^{2^k}\ge \frac{1}{d}\Big| \Big(\frac{p'_{k,\rm rev}(0)}{p_{k,\rm rev}(0)}\Big)\Big|.
\end{equation}

Under bounds (\ref{eqnndcrsd}) and (\ref{eqnndcrsdmx})  these estimates 
 become sharp as $k$ increases, by virtue of Thm. \ref{thvtprsm}; next we  argue informally that they themselves tends to be sharp with a high probability under  random root models.
Indeed, 
\begin{equation}\label{eqratiorcp}
\frac{1}{|x_d|}\le \frac{1}{d}\Big|\frac{p'(c)}{p(c)}\Big|=\frac{1}{d}\Big|\sum_{j=1}^d \frac{1}{c-x_j}\Big|
\end{equation}  
by  virtue of 
(\ref{eqrtrdbndsrev1}), and so the
 approximation  
to the root radius $|x_d|$ is poor if and only if severe cancellation occurs in the summation of the $d$ zeros of $p(x)$, and similarly for the approximation
of  $r_1(c,p)$. Such a cancellation only occurs for a narrow class of polynomials $p(x)$ or, formally, with a low  probability under random root models.
 
 Next we prove, however,  that estimates (\ref{eqrtrdbndsrev1}) and (\ref{eqratio0}) are extremely poor for  worst case inputs.
\begin{theorem}\label{thrtr}
The ratios $|\frac{p(0)}{p'(0)}|$ 
and $|\frac{p_{\rm rev}(0)}{p_{\rm rev}'(0)}|$ are infinite for  $p(x)=x^d-h^d$ and $h\neq 0$, while  $|x_d|=|x_1|=|h|$.
 \end{theorem}
 \begin{proof}
 Observe that 
 the zeros $x_j=h\exp(\frac{(j-1)\sqrt {-1}}{2\pi d})$ of $p(x)=x^d-h^d$
  for $j=1,2,\dots,d$   are the $d$th roots  of unity up to scaling by $h$.
 \end{proof}

Clearly, the problem persists for the root radius $r_d(w,p)$  where  $p'(w)$ and  $p'_{\rm rev}(w)$ vanish;  
 rotation of the variable  $p(x)\leftarrow t(x)=p(ax)$
for $|a|=1$ does not  fix it but  shifts $p(x)\leftarrow t(x)=p(x-c)$
for $c\neq 0$ can fix it, thus {\em enhancing the power of  estimates (\ref{eqrtrdbndsrev1}) and (\ref{eqratio0}).} 

\section{Conclusions}
\label{scnc}

The classical 
DLG root-squaring iterations as well 
as more recent FG iterations of 2001
have been considered purely theoretical constructs  whose application to root-finding has been blocked by
 severe problems of numerical stability. 
 
 By combining Vieta's formulae
 with well-known equation (\ref{eqratio}) for a polynomial $p(x)$ we linked DLG iterations with approximation of the power sums of the zeros of $p(x)$ and of its reverse polynomial and then further observed that
 we can ignore DLG techniques and instead directly reduce root-finding to  
 approximation of the high power sums.

 These computations  can be performed for a black box polynomial
$p(x)$ given by a subroutine for the  evaluation of Newton's inverse ratio
 $-\frac{p'(x)}{p(x)}$ at reasonably many points, which leads to
a  number of additional important benefits, listed in Sec. \ref{sbnfts}.

 We supply some details for   approximation of the power sums by means of specific discretization of Cauchy integrals representing these sums, which can be  alternatively approximated based on  the
  exponentially convergent adaptive version of trapezoid rule in the BOOST
library modified for arbitrary precision.

We comment on numerical stabilization of these computations based on scaling the variable and supply some relevant estimates.

Alternative computation of the power sums based on Newton's identities is quite promising but can only be applied where sufficiently many leading or trailing coefficients 
of $p(x)$ or of its factor are available.

The number  of required coefficients decreases where we approximate an extremal zero of $p(x)$ whose absolute value is not close to the absolute values of any other zero. In that case  approximation of Cauchy integrals representing the power sums
can also be  performed at a lower
 cost. 
 
 Further formal and empirical study should help decide which of variations of these algorithms are more efficient and for which input classes.
 
In Secs. 1.6 and  1.7 we have already listed some tentative research  
directions and challenges.

In Sec. \ref{sapplstrtf} we specify some simple but promising applications of our fast computation of root  radii for 
 a black box polynomial to
 a variation of Lehmer's root-finder and  black box   initialization of 
polynomial  root-finding by means of functional iterations.
 
 The algorithm of \cite{GS22}
 for DLG iterations seems to run faster than the other alternatives, and one should test if this indeed so, and also should test its extension of Sec. 1.3 to  approximation of complex zeros of $p(x)$ by means of descending process. Alternatively, one can try to extend \cite{GS22} to FG  iterations.
   



\medskip

  \medskip

\noindent {\bf Acknowledgements:}
This research has been supported by  NSF Grants  CCF--1563942 and CCF--1733834 and by PSC CUNY Award 
 63677-00-51.



\end{document}